\documentclass[12pt]{article}
\usepackage{amsmath,amssymb,amsthm}
\usepackage{fullpage}

\title{On the accuracy of the Chakrabarti-Hudson approximation to $\pi$}

\author{Mark B. Villarino\footnote{corresponding author; email: mark.villarino@ucr.ac.cr }\\
Escuela de Matem\'atica, Universidad de Costa Rica,\\
10101 San Jos\'e, Costa Rica}

\date{\today}

\newtheorem{lemma}{Lemma}
\newtheorem{thm}{Theorem}
\newtheorem{cor}{Corollary}

\numberwithin{equation}{section}

\makeatletter
\def\section{\@startsection{section}{1}{\z@}{-3.5ex plus -1ex minus
			  -.2ex}{2.3ex plus .2ex}{\large\bf}}
\def\subsection{\@startsection{subsection}{2}{\z@}{-3.25ex plus -1ex
			  minus -.2ex}{1.5ex plus .2ex}{\normalsize\bf}}
\renewcommand{\@dotsep}{200} 
\makeatother

\renewcommand{\geq}{\geqslant}  
\renewcommand{\leq}{\leqslant}  

\newcommand{\half}{{\mathchoice{\thalf}{\thalf}{\shalf}{\shalf}}}
\newcommand{\shalf}{{\scriptstyle\frac{1}{2}}} 
\newcommand{\thalf}{\tfrac{1}{2}} 

\newcommand{\third}{\tfrac{1}{3}}   
\newcommand{\twothirds}{\tfrac{2}{3}} 

\newcommand{\word}[1]{\quad\mbox{#1}\quad} 

\begin{document}

\maketitle

\begin{abstract}
We obtain rigorous upper and lower bounds for the error in the recent
approximation for $\pi$ proposed by \textsc{Chakrabarti \& Hudson}.
\end{abstract}



\textbf{}

\section{Introduction}

In 2003, \textsc{Chakrabarti} and \textsc{Hudson} \cite{CH} proposed a
new formula in the spirit of \textsc{Archi\-me\-des} for the computation
of $\pi$, namely
\begin{equation}
\label{CH}
\pi \approx \frac{1}{30} \{32\pi_n + 4\Pi_{2n} - 6a_n\},
\end{equation}
where $\pi_n$ is the perimeter of an inscribed regular polygon of $n$
sides in a circle of perimeter $\pi$, $\Pi_n$ is the perimeter of a
circumscribed regular polygon of $n$ sides in a circle of perimeter
$\pi$, and $a_n$ is the area of an inscribed regular polygon of $n$
sides in a circle of \emph{area}~$\pi$.

Archimedes, in \emph{The Measurement of the Circle} (MC),
see~\cite{Hea}, obtains his own approximation by starting with the
(evident) inequality
\begin{equation}
\pi_n < \pi < \Pi_n,
\end{equation}
and takes $n = 6, 12, 24, 48, 96$, respectively. He performs a
brilliant \emph{tour de force} of manipulating and rounding the
inequalities resulting from continued fraction expansions of the
square roots which arise in the computations of $\pi_n$ and $\Pi_n$,
whence he obtains his justly famous bounds
\begin{equation}
3+\frac{10}{71} < \pi < 3+\frac{1}{7}.
\end{equation}

It is generally recognized, see \cite{Knorr}, that the extant MC is a
post-Archimedean revision of Archimedes' original and far more
comprehensive treatment of the circle. For example, \textsc{Heron},
\emph{Metrica} I,~32 \cite{Heron}, quotes a theorem of Archimedes from
(MC), which is \emph{not present in the extant version}. The theorem
states that \emph{the area of a circular sector exceeds four-thirds the area
of the greatest inscribable triangle}. In terms of the notation above,
this inequality affirms that:
\begin{equation}
\label{Hero}
\pi > \third (4\pi_{2n} - \pi_n).
\end{equation}

The right hand side of \eqref{Hero} is a \emph{(generalized) convex combination} of
$\pi_{2n}$ and $\pi_n$ since $\frac{4}{3} - \frac{1}{3} = 1$. (We used the adjetive ``generalized" since usually the coefficients in a convex combination are taken to be positive.  We will supress the word ``generalized" from here on, hoping that the reader will keep this comment in mind.)
Moreover, its form argues for the fact that the original treatise of
Archimedes studied convergence-improvement inequalities derived by
means of geometric theory. Unfortunately, such theoretical aspects
apparently proved too subtle for the uses of later commentators such as
Heron.

Observe that the formula of \textsc{Chakrabarti \& Hudson} \eqref{CH},
too, is a convex combination of its terms.

Such convex combinations did not reappear until the sixteenth and
seventeenth centuries in the work of \textsc{Snell}~\cite{Snell} and
\textsc{Huygens}~\cite{Huy}. The former stated and the latter proved
the following upper bound inequality:
\begin{equation}
\pi < \twothirds \pi_n + \third \Pi_n,
\end{equation}
which is a beautiful generalization of Archimedes' original
inequality. Both Snell and Huygens use it in the form
\begin{equation}
\pi \approx \twothirds \pi_n + \third \Pi_n
\end{equation}
and they justifiably extol the convergence-rate improvement produced
by this convex combinations of quantities already computed.

Now, Chakrabarti and Hudson ask what happens if we use the corresponding \emph{areal} convex
combination
\begin{equation}
\pi \approx \twothirds a_n + \third A_n
\end{equation}
to approximate $\pi$, where $A_n$ is the area of a circumscribed
regular polygon of $n$ sides in a circle of area~$\pi$. They prove the
very interesting result that
\begin{equation}
\label{CH1}
\lim_{n\to\infty} \frac{\frac{2}{3}\pi_n + \frac{1}{3}\Pi_n - \pi}
{\frac{2}{3}a_n + \frac{1}{3}A_n - \pi} = \frac{3}{8}\,,
\end{equation}
which shows that \emph{the areas converge much more slowly than the
perimeters to~$\pi$}.

Then they suppress the limit notation, treat \eqref{CH1} \emph{as an
equation for~$\pi$} and when they reduce it algebraically they obtain
the approximative formula~\eqref{CH}. Passing to trigonometric functions, \eqref{CH} becomes:
\begin{equation}
\label{CH2}
\pi \approx \frac{n}{30} \biggl\{ 32 \sin\Bigl( \frac{\pi}{n} \Bigr)
+ 4\tan\Bigl( \frac{\pi}{n} \Bigr) - 3\sin\Bigl( \frac{2\pi}{n} \Bigr)
\biggr\} \equiv \Pi(n).
\end{equation}

It is worthwhile to point out that the procedures producing the improved convex combinations cited above are today instances of ``\textsc{Richardson} extrapolation" (see \cite{Hld}).  In the case of the \textsc{Chakrabarti-Hudson} formula, we can obtain it in the following way.  Put $$x:=a\sin x+b\tan x+c\sin(2x),$$ with unknown coeficients and match the first three terms of the \textsc{Taylor} series, which leads to the linear system:\begin{equation*}
\label{ }
a+2b+c=1\ \ \ \ \frac{a}{6}+\frac{b}{3}-\frac{8c}{6}=0\ \ \ \ \frac{a}{120}+\frac{2b}{15}+\frac{32c}{120}=0.
\end{equation*}The error term then has initial term $\frac{1}{105}x^6$, which, indeed was obtained by \textsc{Chakrabarti-Hudson}.  

Nevertheless, it is of interest to see how the seemingly \emph{ad hoc} method of the authors produced the same result as the modern method.

Finally, the authors offer some numerical studies of the accuracy of
\eqref{CH2}.

The investigation of the authors needs to be completed in several areas.
\begin{itemize}
\item
They do not state whether the approximation \eqref{CH2} is in excess
or in defect\dots\ in fact, we will prove that it is in \emph{excess}.
\item
They do not develop any rigorous error analysis\dots\ no upper bounds
for the error nor lower bounds; we will present such bounds.
\item
They offer a non-standard definition of accuracy, which they call
``precision''. This makes it difficult to compare their numerical
results with standard error studies. We will present the standard
definition of ``correct significant digits'' and reformulate the
discussion of the accuracy of~\eqref{CH2} in light of our error
bounds.
\end{itemize}

Recently, \textsc{M. Szyszkowicz} (see \cite{MS}) exploited the Richardson method to present eighteen  different approximations (!) for $\pi$ of which the most accurate is (in his notation)
\begin{equation}
\label{M12}
M_{12}:=\frac{\sin x\cdot (187+24\cos x-\cos 2x))}{10+90\cos x}=x-\frac{x^9}{17640}-\frac{x^{11}}{226385}-\cdots.
\end{equation}which in in \emph{excess}.  (The formula given in his paper has a misprint...an extra factor ``x" in the numerator which should not be there.)  We point out that the third convergent of the continued fraction expansion of $\frac{\sin x}{x}$ given later on (see \eqref{cf}) gives the approximation\begin{equation}
\label{cf3}
\frac{\sin x \cdot (51+48\cos x+6\cos^2 x)}{80+25\cos x}=x-\frac{x^9}{44100}-\frac{x^{11}}{226380}-\cdots
\end{equation}also in \emph{excess}, and which is of comparable complexity, but is $2\frac{1}{2}$ times more accurate asymptotically.  It would be interesting to develop rigorous error bounds for his eighteen approximation formulas.


\section{Error Analysis}

We will prove the following theorem:

\begin{thm}
If $n\geq 32$ then the following inequality is valid:
\begin{equation}
\label{bound}
\boxed{ \frac{1}{105} \Bigl(\frac{\pi}{n}\Bigr)^6 
< \frac{\Pi(n) - \pi}{\pi}
< \frac{1}{104\frac{7}{10}} \Bigl(\frac{\pi}{n}\Bigr)^6 }
\end{equation}
where the lower-bound constant $\dfrac{1}{105}$ is the best possible.
\end{thm}

An immediate consequence is:

\begin{cor}
The approximation $\pi \approx \Pi(n)$ is in \textbf{excess}.
\qed
\end{cor}

The inequality \eqref{bound} bounds the \emph{relative error} in the
approximation $\Pi(n) \approx \pi$. It is well known that an
approximation has $n$ \emph{correct significant digits} iff the
relative error does not exceed $\dfrac{(\half)}{10^n}$,
see~\cite{Hld}. Therefore, we can say:

\begin{cor}
The approximation $\pi \approx \Pi(n)$ has about $(6\log_{10} n - 1.27)$
correct significant digits.
\qed
\end{cor}

Now we turn to the proof of \eqref{bound}. We will use the
\textsc{MacLaurin} expansions of the functions involved. Define:
\begin{equation}
\label{fx}
f(x) := \frac{\pi}{30x} (32\sin x + 4\tan x - 3\sin 2x).
\end{equation}
Then we see that 
\begin{equation}
\label{fx1}
\Pi(n) = f\Bigl( \frac{\pi}{n} \Bigr).
\end{equation}

Since the sum of a convergent alternating series is bracketed by two
consecutive partial sums if the absolute values of the terms decrease monotonically (which here occurs), we obtain:

\begin{lemma}
The following inequality is valid for all $ \frac{\pi}{32}\geq x> 0$: 
\begin{equation}
\label{32sinx}
32\biggl( x - \frac{x^3}{3!} + \frac{x^5}{5!} - \frac{x^7}{7!} \biggr)
< 32\sin x
< 32\biggl( x - \frac{x^3}{3!} + \frac{x^5}{5!} - \frac{x^7}{7!}
+ \frac{x^9}{9!} \biggr).
\end{equation}
\qed
\end{lemma}

The \textsc{MacLaurin} expansion of the tangent function is not an
alternating series. But we can still use the \textsc{Lagrange} form of
the remainder to obtain:

\begin{lemma}
The following inequality is valid for $0< x \leq \frac{\pi}{32}$:
\begin{multline}
\label{4tanx}
4\biggl( x + \frac{1}{3} x^3 + \frac{2}{15} x^5 + \frac{17}{315} x^7 
+ \frac{62}{2835} x^9 \biggr) < 4\tan x
\\
< 4\biggl( x + \frac{1}{3} x^3 + \frac{2}{15} x^5 + \frac{17}{315} x^7
+ \frac{62}{2835} x^9 + \frac{1}{85} x^{11} \biggr).
\end{multline}
\end{lemma}

\begin{proof}
The \textsc{MacLaurin} expansion of order~$11$ of the tangent function 
is
\begin{equation}
\tan x = x + \frac{1}{3} x^3 + \frac{2}{15} x^5 + \frac{17}{315} x^7 
+ \frac{62}{2835} x^9 + \frac{1}{11!} R(\theta_{11}) x^{11},
\end{equation}
where
\begin{align*}
R(x) := \frac{d^{11}}{dx^{11}}(\tan x) = 256(\tan^2x + 1)
& (155925 \tan^{10}x + 467775 \tan^8x + 509355 \tan^6x
\\
& + 238425 \tan^4x  + 42306 \tan^2x + 1382) 
\end{align*}
and $0 \leq \theta_{11} \leq \frac{\pi}{32}$. The function $R(x)$
monotonically increases in the interval, whence
\begin{equation}
353792 = R(0) < R(\theta_{11}) \leq R\Bigl( \frac{\pi}{32} \Bigr)
= 469223.9941\dots
\end{equation}
and we conclude (since $x \geq 0$) that
\begin{equation}
0 \leq \frac{1}{11!} R(\theta_{11}) x^{11}
\leq \frac{1}{11!} \cdot 469223.9941\dots x^{11}
= \frac{1}{85.06\dots} x^{11} < \frac{1}{85} x^{11}.
\end{equation}
This completes the proof.
\end{proof}

Finally, as in the inequality \eqref{32sinx} we obtain:

\begin{lemma}
The following inequality is valid for all $0< x \leq \frac{\pi}{32}$: 
\begin{multline}
\label{3sin2x}
-3\biggl\{ 2x - \frac{(2x)^3}{3!} + \frac{(2x)^5}{5!}
- \frac{(2x)^7}{7!} + \frac{(2x)^9}{9!} \biggr\} < -3\sin 2x
\\
< -3\biggl\{  2x - \frac{(2x)^3}{3!} + \frac{(2x)^5}{5!}
- \frac{(2x)^7}{7!} + \frac{(2x)^9}{9!} - \frac{(2x)^{11}}{11!}
\biggr\}.
\end{multline}
\qed
\end{lemma}

Now that we have set up the technical inequalities necessary in our
main proof, we enter into its details.

\begin{proof}[Proof of the main theorem]
If we substitute the inequalities \eqref{32sinx}, \eqref{4tanx} and
\eqref{3sin2x} into the formula for $f(x)$, \eqref{fx}, we obtain
\begin{equation}
\frac{1}{105}x^6 + \frac{1}{360}x^8 + \frac{2776}{2338875}x^{10}
< \frac{f(x)}{\pi} - 1
< \frac{1}{105}x^6 + \frac{1}{360}x^8 + \frac{10531}{26507250}x^{10}.
\end{equation}

Now, the left-hand side
\begin{equation}
\frac{1}{105} x^6 + \frac{1}{360} x^8 + \frac{2776}{2338875} x^{10}
= \frac{1}{105} x^6 \biggl\{ 1 + \frac{7}{24} x^2
+ \frac{2776}{22275} x^4 \biggr\} > \frac{1}{105} x^6
\end{equation}
for all positive~$x$. So, we have proven the lower bound \eqref{bound}
and that the constant $\dfrac{1}{105}$ cannot be replaced by a bigger
one, i.e., \emph{it is the best possible}.

For the upper bound, we observe that for
$0 < x \leq \frac{\pi}{32}$,
\begin{align*}
\frac{1}{105} x^6 + \frac{1}{360} x^8 + \frac{10531}{26507250} x^{10}
&= \frac{1}{105} x^6 \biggl\{ 1 + \frac{7}{24} x^2
+ \frac{10531}{252450} x^4 \biggr\}
\\
&< \frac{1}{105} x^6 \biggl\{ 1
+ \frac{7}{24} \Bigl( \frac{\pi}{32} \Bigr)^2
+ \frac{10531}{252450} \Bigl( \frac{\pi}{32} \Bigr)^4 \biggr\}
\\
&= \frac{1}{104.705\dots} x^6 < \frac{1}{104\frac{7}{10}} x^6,
\end{align*}
which is the upper bound presented in~\eqref{bound}.

This, together with \eqref{fx1} complete the proof of the theorem.
\end{proof}


\section{Numerical Error Studies}

Chakrabarti and Hudson present a table of numerical studies of the
error in their approximative formula~\eqref{CH1}. They define: a
number $\alpha$ has \emph{precision}~$n$ if
\begin{equation}
|\alpha - \pi| < \frac{1}{10^n} \,.
\end{equation}
We already pointed out the standard definition (see~\cite{Hld}): an
approximation $\overline{N}$ approximates the true value $N$ with $n$
\emph{correct significant digits} if the positive relative error
\begin{equation}
\frac{|N - \overline{N}|}{N} < \frac{(\half)}{10^n} \,.
\end{equation}
Of course the two definitions will coincide sometimes, and sometimes
not. However, we note that
$$
|\alpha - \pi| < \frac{1}{10^n} 
\iff \frac{|\alpha - \pi|}{\pi} < \frac{1}{\pi 10^n}
< \frac{(\half)}{10^n}
$$
so that \emph{an approximation with precision~$n$ is always correct to $n$
significant digits}. It is sufficient, but not necessary, since if
$$
\frac{1}{\pi 10^n} < \frac{|\alpha - \pi|}{\pi}
< \frac{(\half)}{10^n} \,,
$$
then $\alpha$ \emph{will have precision $n - 1$ but still have $n$ correct
significant digits}.


\section{Earlier approximations}

We can rewrite the approximation $1 \approx \dfrac{f(x)}{\pi}$ as
follows:
\begin{equation}
\label{CH3}
\frac{\sin x}{x} \approx \frac{15\cos x}{2 + 16\cos x - 3\cos^2 x} \,,
\end{equation}
where we know that the approximate value is \emph{smaller} than the true
value by about $\dfrac{x^6}{105}$.

Now, $\dfrac{\sin x}{x}$ has the following \emph{continued fraction
expansion}:
\begin{equation}
\label{cf}
\frac{\sin x}{x}
= 1 - \cfrac{ \frac{1\cdot 2}{1\cdot 3} \sin^2 \frac{x}{2}}
{1 - \cfrac{ \frac{1\cdot 2}{3\cdot 5} \sin^2 \frac{x}{2}}
{1 - \cfrac{ \frac{3\cdot 4}{5\cdot 7} \sin^2 \frac{x}{2}}
{1 - \cfrac{\frac{3\cdot 4}{7\cdot 9} \sin^2 \frac{x}{2}}
{1-\cdots}}}}
\end{equation}
The first few convergents are:
\begin{align}
\frac{p_1}{q_1} &= \frac{2 + \cos x}{3}\,;
&\text{error} &= - \frac{1}{180} x^4 + \cdots  \word{(\textsc{Snell})}
\nonumber \\
\frac{p_2}{q_2} &= \frac{9 + 6\cos x}{14 + \cos x}\,;
&\text{error} &= - \frac{1}{2100} x^6 + \cdots \word{(\textsc{Newton})}
\nonumber \\   
\frac{p_3}{q_3} &= \frac{51 + 48\cos x + 6\cos^2x}{80 + 25\cos x}\,;
&\text{error} &= - \frac{1}{44100} x^8 + \cdots
\end{align}
where we have used $\sin^2 \frac{x}{2} = \half(1 - \cos x)$. (For all
these results, see~\cite{Vah}.) Every one of these approximations is
\emph{larger} than the true value and is the \emph{best possible}.

The Chakrabarti--Hudson approximation \eqref{CH3} has an error about\emph{ twenty times greater} than that of Newton's formula, although both are of order $o(x^6)$, nor is it as simple. Indeed, it has the
formal complexity of the \emph{third} convergent without the latter's
extraordinary accuracy. Nevertheless, it is the only approximation in
\emph{defect}, its accuracy is still quite good, and it is interesting
that such an \emph{ad hoc} derivation produced such an intriguing
approximation.

\subsubsection*{Acknowledgment}

Support from the Vicerrector\'ia de Investigaci\'on of the University
of Costa Rica is acknowledged.



\begin{thebibliography}{8}

\bibitem{CH}
G. Chakrabarti and R. Hudson,
``An improvement of Archimedes' method of approximating $\pi$'',
Int. J. Pure Appl. Math. \textbf{7} (2003), 207--212.

\bibitem{Hea}
T. L. Heath.
\textit{The Works of Archimedes},
Cambridge University Press, Cambridge, 1897.

\bibitem{Heron}
Heron,
\textit{Opera}, 5 volumes (III: \textit{Metrica}, ed. H.~Sch\"one),
Teubner, Leipzig, 1903.

\bibitem{Hld}
F. B. Hildebrandt.
\textit{Introduction to Numerical Analysis},
Dover, New York, 1987.

\bibitem{Huy}
Ch. Huygens,
\textit{De circuli magnitudine inventa},
Elzevier, Leiden, 1654.

\bibitem{Knorr}
W. R. Knorr,
``Archimedes and the Measurement of the Circle: A New Interpretation'',
Arch. Hist. Exact Sci. \textbf{15} (1976), 115--140.

\bibitem{Snell}
W. Snell,
\textit{Cyclometricus},
Leiden, 1621.


\bibitem{MS}
M. Szyszkowicz,
``Approximations of pi and squaring the circle'',
JMEST \textbf{2} (2015), 330--332.











\bibitem{Vah}
K. T. Vahlen,
\textit{Konstruktionen und Approximationen},
Teubner, Leipzig, 1911.

\end{thebibliography}
\end{document}